\documentclass[12pt]{article}

\usepackage{amsmath,amsthm}
\usepackage{amssymb}
\usepackage{color}
\usepackage{breqn}
\usepackage{enumerate}
\usepackage{graphicx}
\usepackage{bm}
\usepackage{bbm}
\usepackage[affil-it]{authblk}
\usepackage{tabu}
\usepackage{bold-extra}
\usepackage[hang,flushmargin]{footmisc} 
\usepackage{hyperref}
\usepackage{mathtools}
 \usepackage{relsize}
 
\newtheorem{theorem}{Theorem}
\newtheorem{corollary}[theorem]{Corollary}
\newtheorem{lemma}[theorem]{Lemma}
\newtheorem{proposition}[theorem]{Proposition}

\newtheorem{claim}[theorem]{Claim}
\newtheorem*{claim*}{Claim}

\theoremstyle{definition}

\newcounter{boldSectionCounter}
\stepcounter{boldSectionCounter}
\newcounter{boldSubsectionCounter}
\stepcounter{boldSubsectionCounter}

\allowdisplaybreaks

\newcommand{\tdt}{
	\times\dots\times
}

\newcommand{\boldSection}[1]{
   \large\begin{center}\noindent {\bfseries{\scshape \S \arabic{boldSectionCounter} #1}}\\[6pt]\end{center}\normalsize
   \stepcounter{boldSectionCounter}
	 \setcounter{boldSubsectionCounter}{1}
}

\newcommand{\tightoverset}[2]{
  \mathop{#2}\limits^{\vbox to -.5ex{\kern-1.15ex\hbox{$#1$}\vss}}}

\newcommand{\dom}{
	\operatorname{dom}
}

\newcommand\restr[2]{{
  \left.\kern-\nulldelimiterspace 
  #1 
  \vphantom{\big|} 
  \right|_{#2} 
}}	

\newcommand\blfootnote[1]{%
  \begingroup
  \renewcommand\thefootnote{}\footnote{#1}%
  \addtocounter{footnote}{-1}%
  \endgroup
}

\newcommand\ex{
	\mathop{\mathbb{E}}
}

\makeatletter
\newcommand*\bcdot{\mathpalette\bigcdot@{0.5}}
\newcommand*\bigcdot@[2]{\mathbin{\vcenter{\hbox{\scalebox{#2}{$\m@th#1\bullet$}}}}}
\makeatother

\makeatletter
\def\blfootnote{\gdef\@thefnmark{}\@footnotetext}
\makeatother	

\def\colon{:}
	
\begin{document}
\begin{center}\Large\noindent{\bfseries{\scshape A note on extensions of multilinear maps defined on multilinear varieties}}\\[24pt]
\normalsize\noindent{\scshape W. T. Gowers\dag and L. Mili\'cevi\'c\ddag}\\[6pt]
\end{center}
\blfootnote{\noindent\dag\ Royal Society 2010 Anniversary Research Professor, University of Cambridge\\
\noindent\ddag\ Mathematical Institute of the Serbian Academy of Sciences and Arts\\\phantom{\dag\ }Email: luka.milicevic@turing.mi.sanu.ac.rs}

\footnotesize
\begin{changemargin}{1in}{1in}
\centerline{\sc{\textbf{Abstract}}} Let $G_1, \dots, G_k$ be finite-dimensional vector spaces over a finite field $\mathbb{F}$. A multilinear variety of codimension $d$ is a subset of $G_1 \tdt G_k$ defined as the zero set of $d$ forms, each of which is multilinear on some subset of the coordinates. A map $\phi$ defined on a multilinear variety $B$ is multilinear if for each coordinate $d$ and all choices of $x_i \in G_i$, $i\not=d$, the restriction map $y \mapsto \phi(x_1, \dots, x_{d-1}, y, x_{d+1}, \dots, x_k)$ is linear where defined. In this note, we show that a multilinear map defined on a multilinear variety of codimension $d$ coincides on a multilinear variety of codimension $d^{O(1)}$ with a multilinear map defined on the whole of $G_1\times\dots\times G_k$.
\end{changemargin}

\boldSection{Introduction}

In~\cite{U4inverse}, the authors proved a quantitative version of the inverse theorem for the Gowers $U^4$ norm over finite fields. The proof depended on a series of results about maps that have bilinear behaviour on subsets of $\mathbb F_p^n$, which included he following theorems. In the statements, $G_1, G_2, H$ are finite-dimensional vector spaces over $\mathbb F_p$ and $\omega = e^{2\pi i/p}$.

\begin{theorem}[Theorem 7.7 in~\cite{U4inverse}]\label{highRankBilExt}Suppose that $r \geq 20d$ and that $\beta : G_1 \times G_2 \to \mathbb{F}_p^d$ is a bilinear map that satisfies $\ex_{x \in G_1, y \in G_2} \omega^{\lambda \cdot \beta(x,y)} \leq p^{-r}$ for all $\lambda \in \mathbb{F}_p^d \setminus \{0\}$. Let $D = \{(x,y) \in G_1 \times G_2 \colon \beta(x,y) = 0\}$. Let $\phi \colon D \to H$ be a bilinear map, in the sense that for each $x \in G_1$, the map $\phi_{x \cdot} \colon \{y \in G_2 : (x,y) \in D\} \to H$ given by $y \mapsto \phi(x,y)$ is linear, and the analogous statement holds for the second coordinate. Then there is bilinear map $\Phi : G_1 \times G_2 \to H$ such that $\Phi(x,y) = \phi(x,y)$ for all $(x,y) \in D$.\end{theorem}

The condition on $\beta$ in the above theorem is equivalent to the statement that the bilinear form $\lambda.\beta$ has rank at least $r$ for every non-zero $\lambda\in\mathbb F_p^d$. Without it, the conclusion is not necessarily true, but the next theorem tells us that if the condition does not hold, then we can pass to small-codimensional subspaces where it does.

\begin{theorem}[Theorem 5.2 in~\cite{U4inverse}]Let $\beta : G_1 \times G_2 \to \mathbb{F}_p^d$ be a bilinear map and let $r$ be a positive integer. Then, there are subspaces $V_1 \leq G_1, V_2 \leq G_2$ of codimension at most $rd$ such that $\ex_{x \in V_1, y \in V_2} \omega^{\lambda \cdot \beta(x,y)} \leq p^{-r}$ for all $\lambda \in \mathbb{F}_p^d \setminus \{0\}$.\end{theorem}

Let us say that a set of the form $\{(x,y) \in G_1 \times G_2 \colon \alpha(x) = 0, \beta(y) = 0, \gamma(x,y) = 0\}$ for linear maps $\alpha: G_1 \to \mathbb{F}_p^{t_1}, \beta : G_2 \to \mathbb{F}_p^{t_2}$ and bilinear map $\gamma : G_1\times G_2 \to \mathbb{F}_p^{t_3}$ is a \emph{bilinear variety of codimension} $t = t_1 + t_2 + t_3$. We may combine the two theorems above into a single result. 

\begin{corollary}\label{bilinearExtension}Let $\beta : G_1 \times G_2 \to \mathbb{F}_p^d$ be a bilinear map, let $D = \{(x,y) \in G_1 \times G_2 : \beta(x,y) = 0\}$, and let $\phi : D \to H$ be a bilinear map in the sense of Theorem \ref{highRankBilExt}. Then there is a bilinear variety $B \subset D$ of codimension $O(d^2)$ and a bilinear map $\Phi \colon G_1 \times G_2 \to H$ such that $\phi$ agrees with $\Phi$ on $B$.\end{corollary}

As we have mentioned, bilinear maps defined on a bilinear variety in general cannot in general be extended to global bilinear maps, so Corollary~\ref{bilinearExtension} is best we can hope for in a qualitative sense. For a simple example of a non-extendable map, take the variety $B = \{(x_1, x_2; y_1, y_2) \in \mathbb{F}_p^2 \times \mathbb{F}_p^2 : x_1 y_1 - x_2 y_2 = 0\}$. We may partition $B$ into sets $Z, B_\lambda$, where $\lambda \in\mathbb{F}_p \setminus \{0\}$, defined by
\begin{align*}Z = &\{(0,0;y_1, y_2) : y_1, y_2 \in \mathbb{F}_p\} \cup \{(x_1, x_2; 0,0) : x_1, x_2 \in \mathbb{F}_p\}\\
&\hspace{2cm}\cup \{(0,x_2;y_1, 0) : x_2, y_1 \in \mathbb{F}_p\} \cup \{(x_1,0;0, y_2) : x_1, y_2 \in \mathbb{F}_p\}\end{align*}
and 
\[B_\lambda = \{(\lambda x, x; y, \lambda y) : x,y \in \mathbb{F}_p \setminus \{0\}\}.\]
Let $f: \mathbb{F}_p\setminus \{0\} \to \mathbb{F}_p$ be any map. Define a map $\phi : B \to \mathbb{F}_p$ by $\phi(x_1, x_2; y_1, y_2) = 0$, when $(x_1, x_2; y_1, y_2) \in Z$, and $\phi(x_1, x_2; y_1, y_2) = f(\lambda) x_2y_1$, when $(x_1, x_2; y_1, y_2) \in B_\lambda$, $\lambda \not= 0$. It is easy to check that $\phi$ is a bilinear map on $B$ for any choice of $f$.\\
\indent To see that $\phi$ cannot be extend to a global bilinear map, it suffices to show that the restriction $\psi \colon \{(x,x) : x \in \mathbb{F}_p\} \to \mathbb{F}_p$ defined by $\psi(x,x) = \phi(x, 1; 1, x)$, cannot be extended to a biaffine map on $\mathbb{F}_p\times\mathbb{F}_p$ for some $f$. Observe that $\psi(x,x) = f(x)$ when $x \not= 0$, and $\psi(0,0) = 0$, so there are $p^{p-1}$ different $\psi$ we may get, while there are only $p^4$ biaffine maps on $\mathbb{F}_p \times \mathbb{F}_p$.\\

The aim of this note is to generalize Corollary~\ref{bilinearExtension} to the multivariate case. Let $\mathbb F$ be a finite field, which we shall regard as fixed, and now let $G_1,\dots,G_k$ be vector spaces over $\mathbb F$. We define a \emph{multilinear variety of codimension $d$} in $G_1 \tdt G_k$ to be a set of the form $\{(x_1, \dots, x_k) \in G_1 \tdt G_k : (\forall i \in [d]) \beta_i(x_{I_i}) = 0\}$, where the maps $\beta_i : \prod_{j \in I_i} G_j \to \mathbb{F}$ are multilinear forms for $i \in [d]$. Our main theorem is the following. 

\begin{theorem}\label{mainThm}For each positive integer $k$ there are constants $C_k, D_k$ such that the following statement holds. Let $B$ be a multilinear variety of codimension $d$ in $G_1\tdt G_k$ and let $\phi \colon B \to H$ be a multilinear map to a vector space $H$ over $\mathbb F$. Then, there is a global multilinear map $\Phi : G_1 \tdt G_k \to H$ such that the set $\Big\{(x_1, \dots, x_k) \in B : \Phi(x_1, \dots, x_k) = \phi(x_1, \dots, x_k)\Big\}$ contains a multilinear variety of codimension at most $C_k d^{D_k}$.\end{theorem}
\noindent Note that the constants $C_k$ and $D_k$ do not depend on the cardinality of $\mathbb F$. 

This result relies crucially on power-type bounds for partition rank in terms of analytic rank, which were independently proved by Janzer~\cite{JanzerRank} and the second author~\cite{LukaRank}. (The relevant definitions and a precise statement of the result will be given at the end of \S 2.) Let us also note that Kazhdan and Ziegler generalized Theorem~\ref{highRankBilExt} in~\cite{KazhdanZiegler1}, but their result, like Theorem~\ref{highRankBilExt}, has the crucial assumption that the domain of the given map is a variety of high rank. However, in higher dimensions, finding a high rank subvariety inside the given variety leads to significantly worse bounds than those in Theorem~\ref{mainThm}.

\noindent\textbf{Acknowledgements.} The second author would like to acknowledge the support of the Ministry of Education, Science and Technological Development of the Republic of Serbia, Grant ON174026.

\boldSection{Preliminaries}
Let $\mathbf{f} = |\mathbb{F}|$. We recall the following notational conventions, definitions and proposition from~\cite{LukaRank}.\\

\noindent\textbf{Notation.}\ In the rest of the paper, we use the following abbreviations in situations where we have many indices appearing in predictable patterns. Given a sequence $x_1, \dots, x_m$, we shall denote it by $x_{[m]}$, and more generally if $I\subset[m]$ then we shall write $x_I$ for the subsequence with indices that run through $I$. We shall do the same for products of the spaces $G_i$ as well: $G_{[k]}$ will stand for $\prod_{i \in [k]} G_i$ and $G_I$ for $\prod_{i \in I} G_i$. For example, instead of writing $\alpha \colon \prod_{i \in I} G_i \to \mathbb{F}$ and $\alpha(x_i \colon i \in I)$, we write $\alpha \colon G_I \to \mathbb{F}$ and $\alpha(x_I)$. Also, we refer to the zero set of a multiaffine map $\alpha \colon G_{[k]} \to H$, where $H$ is a vector space over $\mathbb{F}$, as a \emph{variety}, and the codimension of a variety is $\dim H$. Another convention we adopt is that we write $\ex_x$, without specifying the set from which $x$ is taken, when this causes no confusion. Frequently we shall consider `slices' of sets $S\subset G_{[k]}$, by which we mean sets $S_{x_I} = \{y_{[k]\setminus I} \in G_{[k] \setminus I} \colon (x_I, y_{[k] \setminus I}) \in S\}$, for $I \subset [k], x_I \in G_I$. (Here $(x_I,y_{[k]\setminus I})$ denotes not the concatenation of the two sequences but the sequence $w_{[k]}$, where $w_i=x_i$ when $i\in I$ and $w_i=y_i$ when $i\in [k]\setminus I$.) Occasionally, we might have a single element $z \in G_i$ instead of $x_I$, and in this case we write $S_{i \colon z}$ for the resulting slice, since the direction $i$ is not clear from the notation $z$, unlike in the case of $x_I$. In other words, $S_{i:z}$ is the set $\{y_{[k]\setminus\{i\}}:(z,y_{[k]\setminus\{i\}})\in S\}$ (with a similar interpretation of $(z,y_{[k]\setminus\{i\}})$). Finally, for each vector space $G_i$, fix a dot product. We need this for the characterization of linear forms on $G_i$ -- each linear form $\phi \colon G_i \to \mathbb{F}$ takes the form $\phi(x) = x \cdot u$ for some element $u \in G_i$.\\

Define a graph $\mathcal{G}$ with vertex set $G_{[k]}$ by putting edges between points that differ in a single coordinate. We say that a set $S \subset G_{[k]}$ is \emph{connected} if the induced graph $\mathcal{G}[S]$ is connected. The \emph{diameter} of $S$ is the largest distance between two vertices in the graph $\mathcal{G}[S]$. In the rest of the paper, we fix a non-trivial multiplicative character $\chi \colon \mathbb{F}\to\mathbb{C}$.

\begin{proposition}[One-sided regularity lemma~\cite{LukaRank}]\label{nonzeroConnIntro}Write $c_k = 4(k+1)$. Let $\rho \colon G_{[k]} \to \mathbb{F}$ and $\beta_i \colon G_{I_i} \to \mathbb{F}$ ($i=1,2,\dots,r$) be multilinear maps. Let $\mathcal{I} = \{i \in [r] \colon I_i = [k]\}$. Suppose that
\[\ex_{x_1, \dots, x_k} \chi\Big(\rho(x_{[k]}) - \sum\limits_{i \in \mathcal{I}} \lambda_i \beta_i(x_{[k]})\Big) \leq \eta = \mathbf{f}^{-c_k(r+1)},\]
for any choice of $\lambda \in \mathbb{F}^{\mathcal{I}}$. Then the set of $x_{[k]} \in G_{[k]}$ for which $\rho(x_{[k]})\ne 0$ and $\beta_i(x_{I_i})=0$ for $i=1,2,\dots,r$ is connected and has diameter at most $(2k+1)(2^k-1)$.
\end{proposition}


\begin{corollary}\label{layerPath}Let $\rho, \beta_1,\dots,\beta_r$ be as in Proposition~\ref{nonzeroConnIntro}. Let $x_{[k]}, y_{[k]} \in G_{[k]}$ be such that $\rho(x_{[k]}), \rho(y_{[k]}) \not = 0$ and $\beta_i(x_{I_i}) = \beta_i(y_{I_i}) = 0$ for all $i \in [r]$. Then, there are points $q^0_{[k]}, q^1_{[k]}, \dots, q^s_{[k]} \in G_{[k]}$ with the following properties.
\begin{itemize}
\item Any two consecutive points differ in exactly one coordinate.
\item The first point $q^0_{[k]}$ is equal to $x_{[k]}$, and the last point $q^s_{[k]}$ is equal to $(\lambda_1 y_1, \dots, \lambda_k y_k)$, for some non-zero $\lambda_1, \dots, \lambda_k \in \mathbb{F}$.
\item The number $s$ is at most $(2k+1)(2^k-1)$.
\item We have $\rho(q^0_{[k]}) = \rho(q^1_{[k]}) = \dots = \rho(q^s_{[k]})$ and $\beta_j(q^i_{I_j}) = 0$ for all $i \in [0,s], j \in [r]$. 
\end{itemize}
\end{corollary}

\begin{proof}By Proposition~\ref{nonzeroConnIntro} the set $\{x_{[k]} \in G_{[k]} \colon (\forall i \in [r])\beta_i(x_{I_i}) = 0, \rho(x_{[k]}) \not= 0\}$ is connected and of diameter at most $(2k+1)(2^k-1)$. Hence, there is a sequence $q^0_{[k]}, q^1_{[k]}, \dots, q^s_{[k]} \in G_{[k]}$ that satisfies the first three of the listed properties, $\rho(q^0_{[k]}), \dots, \rho(q^s_{[k]}) \not= 0$ and $\beta_j(q^i_{I_j}) = 0$ for all $i \in [0,s], j \in [r]$. By induction on $t \in [0,s]$, we show that there is a sequence $p^0_{[k]}, p^1_{[k]}, \dots, p^s_{[k]} \in G_{[k]}$ that satisfies the first three of the listed properties, where we relax the first property to allow consecutive points to be equal, and that also satisfies a modified version of the last property, namely that $\rho(p^0_{[k]}) = \rho(p^1_{[k]}) = \dots = \rho(p^t_{[k]}) \not=0$, $\rho(p^{t+1}_{[k]}), \dots, \rho(p^s_{[k]}) \not= 0$, and $\beta_j(p^i_{I_j}) = 0$ for all $i \in [0,s], j \in [r]$. For $t = 0$, we may take $p^i_{[k]} = q^i_{[k]}$. Assume now that the claim holds for some $t < s$, and let $p^0_{[k]}, \dots, p^s_{[k]}$ be the sequence so far. Then, points $p^t_{[k]}$ and $p^{t+1}_{[k]}$ differ in a single coordinate, say $c \in [k]$. Let $\lambda \in \mathbb{F} \setminus \{0\}$ be such that $\rho(p^t_{[k]}) = \lambda \rho(p^{t+1}_{[k]})$. Modify all points $p^{t+1}_{[k]}, \dots, p^{s}_{[k]}$ by multiplying their $c$-coordinate by $\lambda$. It is easy to check that the modified sequence satisfies all the properties.\\
\indent Once we have a sequence for $t =s$, remove points that are equal to their predecessor to finish the proof.\end{proof}

We shall also need to know that the set considered in the results above is necessarily non-empty. To prove this we need two simple lemmas. 

\begin{lemma}[Lemma 11~\cite{LukaRank}]\label{varSize}Let $B \subset G_{[k]}$ be a non-empty variety of codimension $d$. Then $|B| \geq \mathbf{f}^{-kd} |G_{[k]}|$.\end{lemma}

\begin{lemma}[Lovett, Lemma 2.1~\cite{Lovett}]\label{mlBias}Suppose that $\alpha \colon G_{[k]} \to \mathbb{F}$ is a multiaffine form with multilinear part $\alpha^{\text{lin}}$. Then
\[\Big|\ex_{x_{[k]}} \chi(\alpha(x_{[k]}))\Big| \leq \ex_{x_{[k]}} \chi(\alpha^{\text{lin}}(x_{[k]})).\]
\end{lemma}

To save space, given multilinear forms $\beta_1, \dots, \beta_r$ and $\lambda \in \mathbb{F}^r$, we shall write $\lambda \cdot \beta$ for the multilinear form $\sum_{i \in [r]} \lambda_i \beta_i$.

\begin{lemma}\label{nonzeroPtLemma}Let $\rho, \beta_1,\dots,\beta_r \colon G_{[k]} \to \mathbb{F}$ be multilinear forms and let $m \in \mathbb{N}$ be such that for all choices of $\lambda \in \mathbb{F}^r$,
\[\ex_{x_{[k]}} \chi\Big(\rho(x_{[k]}) + (\lambda \cdot \beta)(x_{[k]})\Big) < \mathbf{f}^{- k (r + m)}.\]
Then for any multilinear forms $\gamma_i \colon G_{I_i} \to \mathbb{F}$, $\emptyset \not= I_i \subsetneq [k]$, $i=1,2,\dots,m$, we may find $x_{[k]} \in G_{[k]}$ such that
\begin{itemize}
\item $\rho(x_{[k]}) = 1$,
\item $(\forall i \in [r])\  \beta_i(x_{[k]}) = 0$, and
\item $(\forall i \in [m])\  \gamma_i(x_{I_i}) = 0$.
\end{itemize}
\end{lemma}

\begin{proof}Suppose that, on the contrary, whenever a point $x_{[k]}$ satisfies $\beta_i(x_{[k]}) = 0$ for all $i \in [r]$ and $\gamma_i(x_{I_i}) = 0$ for all $i \in [m]$, then $\rho(x_{[k]}) = 0$. The set of such points is a Bohr variety of codimension at most $k+m$, so by Lemma~\ref{varSize},
\begin{align*}\mathbf{f}^{-k(r+m)}\leq&\ex_{x_{[k]}} \bm{1}\Big((\forall i \in [r])\ \beta_i(x_{[k]}) = 0\ \land\ (\forall i \in [m])\ \gamma_i(x_{I_i}) = 0\Big)\\
 = &\ex_{x_{[k]}} \chi(\rho(x_{[k]})) \bm{1}\Big((\forall i \in [r])\ \beta_i(x_{[k]}) = 0\ \land\ (\forall i \in [m])\ \gamma_i(x_{I_i}) = 0\Big)\\
=&\ex_{x_{[k]}}\, \ex_{\lambda \in \mathbb{F}^r,\,\mu \in \mathbb{F}^m} \chi\Big(\rho(x_{[k]}) + (\lambda \cdot \beta)(x_{[k]}) + \sum_{i \in [m]} \mu_i \gamma_i(x_{I_i})\Big)\\
\leq& \ex_{\lambda \in \mathbb{F}^r,\, \mu \in \mathbb{F}^m}\Big|\ex_{x_{[k]}} \chi\Big(\rho(x_{[k]}) + (\lambda \cdot \beta)(x_{[k]}) + \sum_{i \in [m]} \mu_i \gamma_i(x_{I_i})\Big)\Big|.\\
\end{align*}
By Lemma~\ref{mlBias}, this is at most $\ex_{\lambda \in \mathbb{F}^r}\Big|\ex_{x_{[k]}} \chi\Big(\rho(x_{[k]}) + (\lambda \cdot \beta)(x_{[k]})\Big)\Big|$, which by hypothesis is less than $\mathbf{f}^{- k (r + m)}$. This is a contradiction, so the lemma is proved.\end{proof}

The purpose of the next lemma is to enable us to deduce the value that $\phi$ takes at certain points in a situation where, because $\phi$ is not defined everywhere, one cannot straightforwardly expand and use bilinearity.

\begin{lemma}\label{splittingEqnLemma}Let $U \leq G_1$ and $V \leq G_2$ be subspaces and let $\beta \colon G_1 \times G_2 \to \mathbb{F}^r$ and $\rho \colon G_1 \times G_2 \to \mathbb{F}$ be bilinear. Let $B = \{(x,y) \in U \times V \colon \beta(x,y) = 0\}$ and let $B^0 = \{(x,y) \in B \colon \rho(x,y) = 0\}$. Let $(x,y), (z,w), (u,v) \in B$ be points such that $\rho(x,y) = \rho(z,w) = \rho(u,v) = 1$ and $\rho = 0$ for all other points in $\{x,z,u\} \times \{y,w,v\}$. Let $\phi \colon B^0 \to H$ be a bilinear map. Then, for all $l \in \mathbb{F}$, we have
\begin{align*}\phi(x - lz, ly + w) &=\phi(x - z, y + w) + (l-1)\phi(x  - u, y + v) - (l-1)\phi(z - u, w + v)\\
&\hspace{0.5cm} -(l-1)\phi(x,v) - (l^2-1) \phi(z,y) + (l-1)\phi(u,y) + (l-1)\phi(z,v) - (l-1)\phi(u,w).\end{align*}
Also,
\begin{align}\label{eqSplitOrth}\phi(x - lz, ly + w) &=l\phi(x  - u, y + v) - l\phi(z - u, w + v)\nonumber\\
& \hspace{0.5cm}+\phi(x,w) -l\phi(x,v)  - l^2 \phi(z,y) + l \phi(u,y) + l \phi(z,v) - l\phi(u,w).\end{align}
\end{lemma}

\noindent\textbf{Remark.}~Here and in the rest of the paper, whenever $\phi$ is a map with domain $D$ and we write an expression of the form $\phi(p)$, we are tacitly stating that the point $p$ lies in $D$.

\begin{proof}Note first that our hypotheses imply that all the points where we evaluate $\phi$ do indeed belong to $B^0$. We prove the claim by induction on $l$. For $l = 1$, the claim is easy to check. Assume now that it holds for some $l-1$. Then
\begin{align*}\phi(x - lz, ly + w) = & \phi(x - lz, ly + w + v) - \phi(x, v) + l \phi(z,v) \\
=&\phi(x - (l-1)z - u, ly + w + v) - \phi(z - u, ly + w + v) - \phi(x, v) + l \phi(z,v) \\
=&\phi(x - (l-1)z - u, ly + w + v) - \phi(z - u, w + v) - l \phi(z, y) + l\phi(u,y) - \phi(x, v) + l \phi(z,v) \\
=&\phi(x - (l-1)z - u, (l-1)y + w) + \phi(x - (l-1)z - u, y + v) - \phi(z - u, w + v)\\
 &\hspace{1cm}- l \phi(z, y) + l\phi(u,y) - \phi(x, v) + l \phi(z,v) \\
=&\phi(x - (l-1)z - u, (l-1)y + w) + \phi(x  - u, y + v)- (l-1)\phi(z, y)- (l-1)\phi(z, v)\\
&\hspace{1cm}  - \phi(z - u, w + v) - l \phi(z, y) + l\phi(u,y) - \phi(x, v) + l \phi(z,v) \\
=&\phi(x - (l-1)z, (l-1)y + w)  - (l-1)\phi(u, y) - \phi(u,w) + \phi(x  - u, y + v)\\
&\hspace{1cm}- (l-1)\phi(z, y)- (l-1)\phi(z, v) - \phi(z - u, w + v) \\
&\hspace{1cm} - l \phi(z, y) + l\phi(u,y) - \phi(x, v) + l \phi(z,v) \\
=&\phi(x - (l-1)z, (l-1)y + w) + \phi(x  - u, y + v) - \phi(z - u, w + v)\\
&\hspace{1cm} -\phi(x,v) - (2l-1) \phi(z,y) + \phi(u,y) + \phi(z,v) - \phi(u,w)\\
=&\phi(x - z, y + w) + (l-1)\phi(x  - u, y + v) - (l-1)\phi(z - u, w + v)\\
&\hspace{1cm} -(l-1)\phi(x,v) - (l^2-1) \phi(z,y) + (l-1)\phi(u,y) + (l-1)\phi(z,v) - (l-1)\phi(u,w),\end{align*}
where we applied the induction hypothesis in the last line.\\
\indent To deduce the second equality in the statement, use the first equality with $l = 0$ to write $\phi(x - z, y + w)$ in terms of other summands.\end{proof}

Finally, we shall also need polynomial bounds for partition rank in terms of analytic rank, whose definitions we now recall. Let $\alpha \colon G_{[k]} \to \mathbb{F}$ be a multlinear form.\\

The \emph{partition rank} of $\alpha$, introduced by Naslund in~\cite{Naslund}, is the smallest $r$ such that $\alpha$ can be written in the form $\alpha(x_{[k]}) = \sum_{i \in [r]} \beta_i(x_{I_i}) \gamma_i(x_{[k] \setminus I_i})$, for further multilinear forms $\beta_i \colon G_{I_i} \to \mathbb{F}$ and $\gamma_i \colon G_{[k] \setminus I_i} \to \mathbb{F}$, where $\emptyset \not= I_i \not= [k]$. The \emph{analytic rank} of $\alpha$, introduced by Gowers and Wolf in~\cite{GowersWolf}, is defined to be the quantity $-\log_{\mathbf{f}} \ex_{x_{[k]}} \omega^{\alpha(x_{[k]})}$.\\

When $k=2$, it is straightforward to check that both the partition rank and the analytic rank are equal to the rank of $\alpha$ in the usual linear-algebraic sense. However, when $k\geq 3$ the situation is more complicated, partly because there are many competing algebraic definitions of rank. The fact that partition rank can be bounded in terms of analytic rank was proved by Bhowmick and Lovett in~\cite{BhowmickLovett}, where they obtained Ackermannian bounds. As was very recently proved, one may in fact take polynomial bounds. 


\begin{theorem}[Janzer~\cite{JanzerRank}, Mili\'cevi\'c~\cite{LukaRank}]\label{strongInvThm}For every positive integer $k\geq 2$, there are constants $C = C^{\text{ranks}}_k, D = D^{\text{ranks}}_k > 0$ with the following property. Suppose that $\alpha \colon G_{[k]} \to \mathbb{F}$ is a multilinear form of analytic rank $r$. Then the partition rank of $\alpha$ is at most $C (r^{D} + 1)$.\end{theorem}

Note that the proof in \cite{LukaRank} yields constants $C_k$ and $D_k$ that do not depend on the cardinality of the field $\mathbb{F}$. In the special case of polynomials on a single vector space, this was conjectured by Kazhdan and Ziegler~\cite{KazhdanZiegler1},~\cite{KazhdanZiegler2}.

\boldSection{Extending multilinear maps using one-sided regularity}

When two points $x_{[k]}, y_{[k]} \in G_{[k]}$ differ in a single coordinate, say $d$, we write $(x\ominus y)_{[k]}$ for the point with coordinates $(x\ominus y)_i = x_i=y_i$, when $i \not= d$, and $(x\ominus y)_d = x_d - y_d$. Notice that if $B$ is a multilinear variety, then whenever $x_{[k]}, y_{[k]} \in B$ differ in a single coordinate, the point $x\ominus y$ belongs to $B$ as well. 

\begin{theorem}\label{qrExtension}Let $\rho \colon G_{[k]} \to \mathbb{F}$ and $\beta_i \colon G_{I_i} \to \mathbb{F}$, $i \in [m]$ be multilinear forms. Write $\mathcal{I} = \{i \in [m] \colon I_i = [k]\}$. Let $B = \{x_{[k]} \in G_{[k]} \colon (\forall i \in [m])\, \beta_i(x_{I_i}) = 0\}$ and let $B^0 = \{x_{[k]} \in B \colon \rho(x_{[k]}) = 0\}$. Let $H$ be another $\mathbb{F}$-vector space and let $\phi \colon B^0 \to H$ be a \emph{multilinear map}, i.e.,\ a map such that whenever $x_{[k]}, y_{[k]} \in B^0$ differ in a single coordinate then $\phi(x\ominus y) = \phi(x) - \phi(y)$. Suppose that for each $\lambda \in \mathbb{F}^{\mathcal{I}}$
\begin{equation}\label{qrCondition}\ex_{x_{[k]}} \chi \Big(\rho(x_{[k]}) + \sum_{i \in \mathcal{I}} \lambda_i \beta_i(x_{[k]})\Big) < \frac{1}{2k^2}\mathbf{f}^{-(2k^2 + k + 1) (m + 1)2^{2k+3}}.\end{equation}
Then, for each $z_{[k]} \in B \setminus B^0$ and $h_0 \in H$, there is a unique multilinear map $\phi^{\text{ext}} \colon B \to H$ such that $\restr{\phi^{\text{ext}}}{B^0} = \phi$ and $\phi^{\text{ext}}(z_{[k]}) = h_0$.\end{theorem}

\noindent\textbf{Remark.}~The theorem says that if $\rho$ is sufficiently quasirandom with respect to the other forms $\beta_i$, then we may uniquely extend $\phi$ to the larger variety $B$ that we obtain by removing $\rho$ from the definition of the domain of $\phi$. This observation is crucial and it allows us to avoid strong assumptions such as the domain variety having high rank (as in the result of Kazhdan and Ziegler). 
\medskip

The proof splits up into several stages. We begin by explaining how the map $\phi^{\text{ext}}$ is defined. To simplify the writing slightly, we assume that $\rho(z_{[k]}) = 1$, which we may do without loss of generality. Let $x_{[k]} \in B \setminus B^0$ be given. By Corollary~\ref{layerPath} there is a sequence $z_{[k]} = q^0_{[k]}, q^1_{[k]}, \dots,$ $q^s_{[k]} = (\lambda_1 x_1, \dots, \lambda_k x_k) \in G_{[k]}$ with the properties stated in the conclusion of that corollary, the fourth of which gives us that $\rho(q^s_{[k]})=1$ and therefore that $\rho(x_{[k]})=\prod_{i\in[k]}\lambda_i^{-1}$. For an integer $\mathbf{s}$, we call a sequence that satisfies the first, second and fourth properties of the corollary \emph{$\mathbf{s}$-good} if $s \leq \mathbf{s}$. In particular, the corollary says that there is always a $(2k+1)(2^k-1)$-good sequence.\\
\indent Assume for a moment that $\phi^{\text{ext}} \colon B \to H$ is a multilinear map that extends $\phi$. Then, since each $(q^{i+1} \ominus q^i)_{[k]} \in B^0$, we must have 
\begin{align*}\phi^{\text{ext}}(x_{[k]}) = &\Big(\prod_{i \in [k]} \lambda_i^{-1}\Big) \phi^{\text{ext}}(q^s_{[k]})\\
= &\rho(x_{[k]}) \phi^{\text{ext}}(q^s_{[k]})\\
= &\rho(x_{[k]}) \Big(\phi^{\text{ext}}(q^s_{[k]} \ominus q^{s-1}_{[k]}) + \phi^{\text{ext}}(q^{s-1}_{[k]})\Big)\\
= &\rho(x_{[k]}) \Big(\phi^{\text{ext}}(q^s_{[k]} \ominus q^{s-1}_{[k]}) + \dots + \phi^{\text{ext}}(q^1_{[k]} \ominus q^0_{[k]}) + \phi^{\text{ext}}(q^0_{[k]})\Big)\\
= &\rho(x_{[k]}) \Big(\phi(q^s_{[k]} \ominus q^{s-1}_{[k]}) + \dots + \phi(q^1_{[k]} \ominus q^0_{[k]}) + h_0\Big).\end{align*}
From this we see that if $\phi^{\text{ext}}$ exists, it has to be unique. 

We use this observation to define the map $\phi^{\text{ext}}$. For each $x_{[d]} \in B \setminus B^0$, we 
use Corollary~\ref{layerPath} to choose a sequence $q^0_{[k]} = z_{[k]}, q^1_{[k]}, q^2_{[k]}, \dots, q^s_{[k]} = (\lambda_1 x_1, \dots, \lambda_k x_k)$ in $B\setminus B^0$ such that $\rho$ is equal at all points, any two consecutive points differ in exactly one coordinate, and $\lambda_1, \dots, \lambda_k$ are non-zero elements of  $\mathbb{F}$ and $s\leq\mathbf{s}=(2k+1)(2^k-1)+1$. (The addition of 1 to the bound in Corollary~\ref{layerPath} is intentional here: it will simplify the proof that the map $\phi^{\text{ext}}$ we are defining is multilinear.) We then take $\phi^{\text{ext}}(x_{[k]})$ to be 
\begin{equation}\label{phiDef}\rho(x_{[k]}) \Big(\phi(q^s_{[k]} \ominus q^{s-1}_{[k]}) + \dots + \phi(q^1_{[k]} \ominus q^0_{[k]}) + h_0\Big),\end{equation}
If $x_{[d]} \in B^0$, then we simply set $\phi^{\text{ext}}(x_{[k]}) = \phi(x_{[k]})$. 
\medskip

It remains to show that $\phi^{\text{ext}}$ is well-defined and multilinear.\\

\noindent\textbf{3.1. The extension map is well-defined.} 
\smallskip

Let $q^0_{[k]} = z_{[k]},q^1_{[k]}, \dots, q^s_{[k]} = (\lambda_1 x_1, \dots, \lambda_k x_k)$ and $p^0_{[k]} = z_{[k]}, p^1_{[k]}, \dots, p^t_{[k]} = (\mu_1 x_1, \dots, \mu_k x_k)$ be two $\mathbf{s}$-good sequences. In particular, $\prod_{i \in [k]} \lambda_i = \prod_{i \in [k]} \mu_i \not= 0$. We need to show that
\[\phi(q^s_{[k]} \ominus q^{s-1}_{[k]}) + \dots + \phi(q^1_{[k]} \ominus q^0_{[k]}) + \phi(p^0_{[k]} \ominus p^1_{[k]}) + \dots + \phi(p^{t-1}_{[k]} \ominus p^t_{[k]}) = 0.\]
As a slight digression, we note that if $\phi$ were a global multilinear map, then this would be trivial to prove, since $\phi(q^s_{[k]} \ominus q^{s-1}_{[k]})$ could be split as $\phi(q^s_{[k]}) - \phi(q^{s-1}_{[k]})$, and so on, and $\phi(q^s_{[k]}) = \phi(p^t_{[k]})$. We mimic this proof, by using Lemma~\ref{nonzeroPtLemma} to find a point `orthogonal' to the sequence $q_{[k]}^i$. First we prove the following claim that exploits the properties of such a point (and explains the meaning of `orthogonality' we have in mind).\\ 

In the proof below, and in subsequent arguments, when we write an expression of the form $\Bigl((a_i)_{i\in F}, (b_i)_{i\in E\setminus F}\Bigr)$, it should be understood as the sequence $(c_i)_{i\in E}$ such that $c_i=a_i$ when $i\in F$ and $c_i=b_i$ when $i\in E\setminus F$. 

\begin{proposition}\label{manipProp}Let $q^0_{[k]} = z_{[k]}, q^1_{[k]}, \dots, q^s_{[k]} = (\lambda_1 x_1, \dots, \lambda_k x_k)$ be an $\mathbf{s}$-good sequence and let $\nu_1, \dots, \nu_k \in \mathbb{F}$ be non-zero scalars such that $\prod_{i \in [k]} \nu_i \cdot \prod_{i \in [k]} \lambda_i = 1$. Let $e_{[k]} \in G_{[k]}$ be a point that satisfies the conditions
\begin{itemize}
\item $\rho(e_{[k]}) = -1$,
\item $(\forall \emptyset\not= I \subsetneq [k])(\forall i \in [0,s])\  \rho(e_I, q^i_{[k] \setminus I}) = 0$,
\item $(\forall i \in [0,s])(\forall j \in [m])(\forall \emptyset \not=J \subset I_j)\  \beta_j(e_J, q^i_{I_j\setminus J}) = 0$.
\end{itemize}
Then
\begin{align*}&\phi(q^s_{[k]} \ominus q^{s-1}_{[k]}) + \dots + \phi(q^1_{[k]} \ominus q^0_{[k]}) = \Big(\prod_{i \in [k]} \lambda_i\Big) \phi\Big(x_1 + \nu_1e_1, \dots, x_k + \nu_ke_k\Big) - \phi\Big(z_1 + \nu_1\lambda_1e_1, \dots,  z_k + \nu_k\lambda_ke_k\Big)\\
&\hspace{8cm}-\sum_{\emptyset \not=I \subsetneq [k]} \Big(\prod_{i \in [k]} \lambda_i\Big)\phi\Big((\nu_{i}e_{i})_{i \in I}, (x_{i})_{i \in [k] \setminus I}\Big)\\
&\hspace{8cm}+\sum_{\emptyset \not=I \subsetneq [k]} \phi\Big((\lambda_{i}\nu_{i}e_{i})_{i \in I}, (z_{i})_{i \in [k] \setminus I}\Big).\\
\end{align*}
\end{proposition}

\begin{proof}Suppose that $q^{i+1}_{[k]}$ and $q^{i}_{[k]}$ differ in coordinate $d$. Then
\begin{align*}&\phi((q^{i+1} \ominus q^i)_{[k]}) = \phi\Big(q^i_{[d-1]}, q^{i+1}_d-q^i_d, q^i_{[d+1,k]}\Big) \\
&\hspace{1cm}= \phi\Big(q^i_1 + \nu_1\lambda_1e_1, q^i_{[2, d-1]}, q^{i+1}_d-q^i_d, q^i_{[d+1,k]}\Big) - \phi\Big(\nu_1\lambda_1e_1, q^i_{[2, d-1]}, q^{i+1}_d-q^i_d, q^i_{[d+1,k]}\Big)\\
&\hspace{1cm}=\phi\Big(q^i_1 + \nu_1\lambda_1e_1, q^i_{[2, d-1]}, q^{i+1}_d-q^i_d, q^i_{[d+1,k]}\Big) - \phi\Big(\nu_1\lambda_1e_1, q^{i+1}_{[2, k]}\Big) + \phi\Big(\nu_1\lambda_1e_1, q^{i}_{[2, k]}\Big)\\
&\hspace{1cm}=\phi\Big(q^i_1 + \nu_1\lambda_1e_1, q^i_2 + \nu_2\lambda_2e_2, q^i_{[3, d-1]}, q^{i+1}_d-q^i_d, q^i_{[d+1,k]}\Big) - \phi\Big(q^i_1 + \nu_1\lambda_1e_1, \nu_2\lambda_2e_2, q^{i+1}_{[3, k]}\Big)\\
&\hspace{3cm} + \phi\Big(q^i_1 + \nu_1\lambda_1e_1, \nu_2\lambda_2e_2, q^{i}_{[3, k]}\Big) - \phi\Big(\nu_1\lambda_1e_1, q^{i+1}_{[2, k]}\Big) + \phi\Big(\nu_1\lambda_1e_1, q^{i}_{[2, k]}\Big).\\
\end{align*}
Repeating this argument once for each coordinate apart from the $d$th and using the fact that $q^i_j=q^{i+1}_j$ whenever $j\ne d$, we arrive at the expression 
\begin{align*}
\phi\Big(&q^i_1 + \nu_1\lambda_1e_1, \dots, q^i_{d-1} + \nu_{d-1}\lambda_{d-1}e_{d-1}, q^{i+1}_d-q^i_d, q^i_{d+1} + \nu_{d+1}\lambda_{d+1}e_{d+1}, \dots, q^i_k + \nu_k\lambda_ke_k\Big)\\
&\hspace{3cm} - \sum_{j \in [d-1]} \phi\Big(q^{i+1}_1 + \nu_1\lambda_1e_1, \dots, q^{i+1}_{j-1} + \nu_{j-1}\lambda_{j-1}e_{j-1}, \nu_j\lambda_je_j, q^{i+1}_{[j+1, d-1]}, q^{i+1}_d, q^{i+1}_{[d+1, k]}\Big)\\
&\hspace{3cm} +  \sum_{j \in [d-1]}\phi\Big(q^{i}_1 + \nu_1\lambda_1e_1, \dots, q^{i}_{j-1} + \nu_{j-1}\lambda_{j-1}e_{j-1}, \nu_j\lambda_je_j, q^{i}_{[j+1, d-1]}, q^i_d, q^{i}_{[d+1, k]}\Big)\\
&\hspace{3cm} - \sum_{j \in [d+1, k]} \phi\Big(q^{i+1}_1 + \nu_1\lambda_1e_1, \dots, q^{i+1}_{d-1} + \nu_{d-1}\lambda_{d-1}e_{d-1}, q^{i+1}_d, q^{i+1}_{d+1} + \nu_{d+1}\lambda_{d+1}e_{d+1}, \dots,\\
&\hspace{7cm}q^{i+1}_{j-1} + \nu_{j-1}\lambda_{j-1}e_{j-1}, \nu_j\lambda_je_j,  q^{i+1}_{[j+1, k]}\Big)\\
&\hspace{3cm}+ \sum_{j \in [d+1, k]} \phi\Big(q^{i}_1 + \nu_1\lambda_1e_1, \dots, q^{i}_{d-1} + \nu_{d-1}\lambda_{d-1}e_{d-1}, q^{i}_d, q^{i}_{d+1} + \nu_{d+1}\lambda_{d+1}e_{d+1}, \dots,\\
&\hspace{7cm}q^{i}_{j-1} + \nu_{j-1}\lambda_{j-1}e_{j-1}, \nu_j\lambda_je_j,  q^{i}_{[j+1, k]}\Big).\\
\end{align*}
Expanding this out gives
\begin{align*}\phi\Big(&q^{i+1}_1 + \nu_1\lambda_1e_1, \dots, q^{i+1}_{d-1} + \nu_{d-1}\lambda_{d-1}e_{d-1}, q^{i+1}_d + \nu_{d}\lambda_{d}e_{d}, q^{i+1}_{d+1} + \nu_{d+1}\lambda_{d+1}e_{d+1}, \dots, q^{i+1}_k + \nu_k\lambda_ke_k\Big)\\
&\hspace{1cm}-\phi\Big(q^i_1 + \nu_1\lambda_1e_1, \dots, q^i_{d-1} + \nu_{d-1}\lambda_{d-1}e_{d-1}, q^{i}_d + \nu_{d}\lambda_{d}e_{d}, q^i_{d+1} + \nu_{d+1}\lambda_{d+1}e_{d+1}, \dots, q^i_k + \nu_k\lambda_ke_k\Big)\\
&\hspace{3cm} - \sum_{\emptyset \not= I \subset [k] \setminus \{d\}} \phi\Big((\lambda_{j}\nu_{j}e_{j})_{j \in I}, (q^{i+1}_{j})_{j \in [k] \setminus I}\Big)\\
&\hspace{3cm} + \sum_{\emptyset \not= I \subset [k] \setminus \{d\}}\phi\Big((\lambda_{j}\nu_{j}e_{j})_{j \in I}, (q^{i}_{j})_{j \in [k] \setminus I}\Big)\\
\end{align*}
To see why, note that the first term on the left-hand side expands to the first two terms on the right-hand side. And after that, each set $I$ arises from the expansion of the $j$th summand in one of the sums on the left-hand side only when $j=\max I$. 
\smallskip

Using this, and writing $d_i \in [k]$ for the direction where $q^i_{[k]}$ and $q^{i-1}_{[k]}$ differ for $i \in [s]$, we obtain a telescoping sum from the first two terms, and therefore find that
\begin{align*}&\phi(q^s_{[k]} \ominus q^{s-1}_{[k]}) + \dots + \phi(q^1_{[k]} \ominus q^0_{[k]}) = \phi\Big(q^{s}_1 + \nu_1\lambda_1e_1, \dots, q^{s}_k + \nu_k\lambda_ke_k\Big) - \phi\Big(q^{0}_1 + \nu_1\lambda_1e_1, \dots,  q^{0}_k + \nu_k\lambda_ke_k\Big)\\
& \hspace{8cm}-\sum_{\emptyset \not=I \subsetneq [k]} \sum_{\substack{i \in [1,s]\\d_i \notin I}} \phi\Big((\lambda_{j}\nu_{j}e_{j})_{j \in I}, (q^{i}_{j})_{j \in [k] \setminus I}\Big)\\
&\hspace{8cm}+ \sum_{\emptyset \not=I \subsetneq [k]} \sum_{\substack{i \in [0,s-1]\\d_{i+1} \notin I}} \phi\Big((\lambda_{j}\nu_{j}e_{j})_{j \in I}, (q^{i}_{j})_{j \in [k] \setminus I}\Big),\\
&\hspace{6cm}=\phi\Big(q^{s}_1 + \nu_1\lambda_1e_1, \dots, q^{s}_k + \nu_k\lambda_ke_k\Big) - \phi\Big(q^{0}_1 + \nu_1\lambda_1e_1, \dots,  q^{0}_k + \nu_k\lambda_ke_k\Big)\\
&\hspace{8cm}-\sum_{\emptyset \not=I \subsetneq [k]}\phi\Big((\lambda_{i}\nu_{i}e_{i})_{i \in I}, (q^{s}_{i})_{i \in [k] \setminus I}\Big)\\
&\hspace{8cm}+\sum_{\emptyset \not=I \subsetneq [k]}\phi\Big((\lambda_{i}\nu_{i}e_{i})_{i \in I}, (q^{0}_{i})_{i \in [k] \setminus I}\Big),\\
\end{align*}
and the claim follows after recalling that $q^0_{[k]} = z_{[k]}$ and $q^s_{[k]} = (\lambda_1 x_1, \dots, \lambda_k x_k)$.\end{proof}

To complete the proof that $\phi^{\text{ext}}$ is well defined, we shall need a point $e_{[k]}$ with slightly stronger properties than the ones used in Proposition \ref{manipProp}. The first property is the same, the second and third are the same but now for two $s$-good sequences rather than just one, and the fourth is new.

\begin{proposition}\label{orthoStrong}Given a point $z_{[k]}$ and $\mathbf{s}$-good sequences $q^0_{[k]} = z_{[k]}, q^1_{[k]},\dots, q^s_{[k]}$ and $p^0_{[k]} = z_{[k]}, p^1_{[k]}, \dots, p^t_{[k]}$, there is a point $e_{[k]}$ that satisfies the following conditions.
\begin{itemize}
\item[\textbf{(i)}] $\rho(e_{[k]}) = -1$.
\item[\textbf{(ii)}]  $(\forall \emptyset \not= I \subsetneq [k])(\forall i \in [0,s])\ \rho(e_I; q^i_{[k] \setminus I}) = 0$ and $(\forall \emptyset\not= I \subsetneq [k])(\forall i \in [0,t])\ \rho(e_I; p^i_{[k] \setminus I}) = 0$.
\item[\textbf{(iii)}]  $(\forall i \in [0,s])(\forall j \in [m])(\forall \emptyset \not=J \subset I_j)\ \beta_j(e_J, q^i_{I_j\setminus J}) = 0$ and $(\forall i \in [0,t])(\forall j \in [m])(\forall \emptyset \not=J \subset I_j)\ \beta_j(e_J, p^i_{I_j\setminus J}) = 0$.
\item[\textbf{(iv)}]  For all pairs of distinct coordinates $c_1, c_2 \in [k]$ and all $\lambda_{[k] \setminus \{c_1, c_2\}} \in (\mathbb{F} \setminus \{0\})^{[k] \setminus \{c_1, c_2\}}, \mu \in \mathbb{F}^{\mathcal{I}_{c_1, c_2}}$,
\[\ex_{y_{c_1}, y_{c_2}} \chi\Big(\rho(y_{c_1}, y_{c_2}, (z_j - \lambda_j e_j)_{j \in [k] \setminus \{c_1, c_2\}}) - \sum_{i \in \mathcal{I}_{c_1, c_2}} \mu_i\beta_i(y_{c_1}, y_{c_2}, (z_j - \lambda_j e_j)_{j \in I_i \setminus \{c_1, c_2\}})\Big)\]
is at most $\mathbf{f}^{-(m+1) 2^{k+2}}$, where $\mathcal{I}_{c_1, c_2} = \{i \in [m] \colon c_1,c_2 \in I_i\}$.
\end{itemize}\end{proposition} 


\begin{proof}We begin the proof by using Lemma~\ref{nonzeroPtLemma} to find at least one point that satisfies properties \textbf{(i)}, \textbf{(ii)} and \textbf{(iii)}. To achieve this, we consider the following multilinear forms.
\begin{itemize}
\item For each proper non-empty subset $I\subsetneq[k]$ and each $i\in[0,s]$ we take the form that maps $x_{[k]}$ to $\rho(x_I,q^i_{[k]\setminus I})$.
\item For each proper non-empty subset $I\subsetneq[k]$ and each $i\in[0,t]$ we take the form that maps $x_{[k]}$ to $\rho(x_I,p^i_{[k]\setminus I})$.
\item For each $i\in[0,s]$, each $j\in[m]$, and each non-empty proper subset $J\subset I_j$, we take the form that maps $x_{[k]}$ to $\beta_j(x_J,q^i_{I_j\setminus J})$.
\item For each $i\in[0,t]$, each $j\in[m]$, and each non-empty proper subset $J\subset I_j$, we take the form that maps $x_{[k]}$ to $\beta_j(x_J,p^i_{I_j\setminus J})$.
\end{itemize} 
Assumption~\eqref{qrCondition} of Theorem \ref{qrExtension} implies that for all $\lambda \in \mathbb{F}^{\mathcal{I}}$,
\[\ex_{x_{[k]}} \chi\Big(\rho(x_{[k]}) - \sum_{i \in \mathcal{I}} \lambda_i\beta_i(x_{[k]})\Big) < \mathbf{f}^{-k(m + 1)\mathbf{s}2^{k+1}},\]
where $\mathcal{I} = \{i \in [m] \colon I_i = [k]\}$. Therefore, by Lemma~\ref{nonzeroPtLemma} we have at least one point $x_{[k]}$ which evaluates to zero under all these (after suitable projections) and $\rho(x_{[k]}) = -1$.
But the set of such points is a non-empty variety of codimension at most $(m + 1)\mathbf{s}2^{k+1} + 1$, so by Lemma~\ref{varSize}, there are at least $\mathbf{f}^{-k (m + 1)\mathbf{s}2^{k+1} - k} |G_{[k]}|$ of them.\\
\indent On the other hand, for each $c_1, c_2 \in [k]$, $\mu \in \mathbb{F}^{\mathcal{I}_{c_1, c_2}}$, we have
\begin{align*}\ex_{x_{[k] \setminus \{c_1, c_2\}}}\Big(\ex_{y_{c_1}, y_{c_2}} &\chi\Big(\rho(y_{c_1}, y_{c_2}, x_{[k] \setminus \{c_1, c_2\}}) -  \sum_{i \in \mathcal{I}_{c_1, c_2}} \mu_i\beta_i(y_{c_1}, y_{c_2}, x_{I_i \setminus \{c_1, c_2\}})\Big)\Big)\\
= &\Big|\ex_{x_{[k] \setminus \{c_1, c_2\}}, y_{c_1}, y_{c_2}} \chi\Big(\rho(y_{c_1}, y_{c_2}, x_{[k] \setminus \{c_1, c_2\}}) -  \sum_{i \in \mathcal{I}_{c_1, c_2}} \mu_i\beta_i(y_{c_1}, y_{c_2}, x_{I_i \setminus \{c_1, c_2\}})\Big)\Big|,\\
\end{align*}
since the inner expectation on the left-hand side is always a nonnegative real. 
By Lemma~\ref{mlBias}, the right-hand side is at most 
\[\ex_{x_{[k]}} \chi\Big(\rho(x_{[k]}) -  \sum_{i \in \mathcal{I}} \mu_i\beta_i(x_{[k]})\Big),\]
which, using assumption~\eqref{qrCondition} of Theorem \ref{qrExtension} again, is at most 
\[\frac{1}{2k^2}\mathbf{f}^{-k (m + 1)\mathbf{s}2^{k+1}-(m+1) 2^{k+2} - m - k}.\]
From this we deduce that the set $X_{c_1,c_2} \subset G_{[k] \setminus \{c_1, c_2\}}$ of points $x_{[k] \setminus \{c_1, c_2\}}$ such that for some $\mu \in \mathbb{F}^{\mathcal{I}_{c_1, c_2}}$
\[\ex_{y_{c_1}, y_{c_2}} \chi\Big(\rho(y_{c_1}, y_{c_2}; x_{[k] \setminus \{c_1, c_2\}}) -  \sum_{i \in \mathcal{I}_{c_1, c_2}} \mu_i\beta_i(y_{c_1}, y_{c_2}; x_{I_i \setminus \{c_1, c_2\}})\Big) > \mathbf{f}^{-(m+1) 2^{k+2}}\]
has size $|X_{c_1,c_2}| \leq \frac{1}{2k^2}\mathbf{f}^{-k (m + 1)\mathbf{s}2^{k+1} - k} |G_{[k]}|$. Thus, there is a choice of $e_{[k]}$ such that the properties \textbf{(i)}, \textbf{(ii)} and \textbf{(iii)} hold and for each distinct $c_1, c_2 \in [k]$ and each $\lambda \in (\mathbb{F} \setminus \{0\})^{[k] \setminus \{c_1, c_2\}}$, the sequence $(z_i - \lambda_i e_i \colon i \in [k] \setminus \{c_1, c_2\})$ does not belong to $X_{c_1,c_2}$, which completes the proof.\end{proof}

Next, we exploit the property \textbf{(iv)} to understand how the values of $\phi(z_1 + \lambda_1 e_1, \dots, z_k + \lambda_k e_k)$ are related for different values of $\lambda_{[k]} \in (\mathbb{F}\setminus \{0\})^{[k]}$.

\begin{proposition}\label{zeCancelation}Suppose that $z_{[k]}$ and $e_{[k]}$ have the properties listed in Proposition~\ref{orthoStrong}. Then, for any $\tau, \sigma \in \mathbb{F}^k$ such that $\prod_{i \in [k]} \tau_i =  \prod_{i \in [k]} \sigma_i  = 1$, we have
\begin{align*}\phi\Big(z_1 + \tau_1e_1, \dots,  z_k + \tau_ke_k\Big) &- \sum_{\emptyset \not=I \subsetneq [k]} \phi\Big((\tau_{i}e_{i})_{i \in I}, (z_{i})_{i \in [k] \setminus I}\Big),\\
&=\phi\Big(z_1 +\sigma_1 e_1, \dots,  z_k + \sigma_ke_k\Big) -\sum_{\emptyset \not=I \subsetneq [k]} \phi\Big((\sigma_i e_{i})_{i \in I}, (z_{i})_{i \in [k] \setminus I}\Big).\end{align*} 
\end{proposition}

\begin{proof}It suffices to prove the claim for the case when $\sigma_i = \tau_i$ for $i \in [k] \setminus \{c_1, c_2\}$, for some pair of coordinates $c_1, c_2$, and $\sigma_{c_1} = \delta\tau_{c_1}, \sigma_{c_2} = \eta\tau_{c_2}$, where $\delta \eta = 1$. We shall abuse notation and write $e_i$ instead of $\tau_i e_i$: since the point $(\tau_1 e_1, \dots, \tau_k e_k)$ satisfies the same properties as $e_{[k]}$, this does not affect the correctness of the proof. Also, by symmetry, we may assume without loss of generality that $c_1 = 1$ and $c_2 = 2$. Write $\theta \colon G_{[2]} \to H$ for the map $\theta(x,y) = \phi(x, y, z_3 + e_3, \dots,  z_k + e_k)$. The claim now reduces to showing that
\[\theta(z_1 + e_1, z_2 + e_2) - \theta(e_1, z_2) - \theta(z_1, e_2) = \theta(z_1 + \delta e_1, z_2 + \eta e_2) - \delta\theta(e_1, z_2) - \eta\theta(z_1, e_2).\]
By property \textbf{(iv)} of Proposition~\ref{orthoStrong} and by Lemma~\ref{nonzeroPtLemma} there are $u \in G_1, v \in G_2$ such that $\rho(u, v, z_3 + e_3, \dots,  z_k + e_k) = 1$, and all other values of maps $\rho, \beta_{[m]}$ at points among $\{z_1, e_1, u\} \times \{z_2, e_2, v\} \times \{(z_3 + e_3, \dots,  z_k + e_k)\}$, involving $u$ or $v$, are zero. Therefore, by Lemma~\ref{splittingEqnLemma} (using statement~\eqref{eqSplitOrth} of the lemma for the second and fourth equalities) 
\begin{align*}\theta(z_1 + \delta e_1, z_2 + \eta e_2)& - \delta\theta(e_1, z_2) - \eta\theta(z_1, e_2)\\
&=\eta\theta(z_1 - \delta (-e_1), \delta z_2 + e_2) - \delta\theta(e_1, z_2) - \eta\theta(z_1, e_2)\\
&=\eta \Big(\delta \theta(z_1 - u, z_2 + v) - \delta\theta(-e_1 - u, e_2 + v)\\
& \hspace{1cm}+\theta(z_1,e_2) -\delta\theta(z_1,v)  - \delta^2 \theta(-e_1,z_2) + \delta \theta(u,z_2) + \delta \theta(-e_1,v) - \delta\theta(u,e_2)\Big)\\
&\hspace{2cm} - \delta\theta(e_1, z_2) - \eta\theta(z_1, e_2)\\
&=\theta(z_1 - u, z_2 + v) - \theta(-e_1 - u, e_2 + v) - \theta(z_1,v) + \theta(u,z_2) + \theta(-e_1,v) - \theta(u,e_2)\\
&=\theta(z_1 + e_1, z_2 + e_2) - \theta(e_1, z_2) - \theta(z_1, e_2),
\end{align*}
as desired.\end{proof}

We now return to the proof that $\phi^{\text{ext}}$ is well-defined. Recall that $q^0_{[k]} = z_{[k]}, \dots, q^s = (\lambda_1 x_1, \dots, \lambda_k x_k)$ and $p^0_{[k]} = z_{[k]}, \dots, p^t = (\mu_1 x_1, \dots, \mu_k x_k)$ are two $\mathbf{s}$-good sequences. Apply Proposition~\ref{orthoStrong} to find a point $e_{[k]} \in G_{[k]}$ that has properties described in that proposition. The assumptions of Proposition~\ref{manipProp} are satisfied. Applying the proposition twice with $\nu_i = \lambda_i^{-1}$, we obtain
\begin{align*}\phi(q^s_{[k]} \ominus q^{s-1}_{[k]}) + \dots + \phi(q^1_{[k]} \ominus q^0_{[k]}) = \Big(\prod_{i \in [k]} \lambda_i\Big)& \phi\Big(x_1 + \nu_1 e_1, \dots, x_k + \nu_k e_k\Big) - \phi\Big(z_1 + \nu_1 \lambda_1 e_1, \dots,  z_k + \nu_k \lambda_k e_k\Big)\\
&-\sum_{\emptyset \not=I \subsetneq [k]} \Big(\prod_{i \in [k]} \lambda_i\Big)\phi\Big((\nu_i e_{i})_{i \in I}, (x_{i})_{i \in [k] \setminus I}\Big)\\
&+\sum_{\emptyset \not=I \subsetneq [k]} \phi\Big((\nu_i \lambda_{i}e_{i})_{i \in I},(z_{i})_{i \in [k] \setminus I}\Big),
\end{align*}
and 
\begin{align*}\phi(p^t_{[k]} \ominus p^{t-1}_{[k]}) + \dots + \phi(p^1_{[k]} \ominus p^0_{[k]}) = \Big(\prod_{i \in [k]} \mu_i\Big) &\phi\Big(x_1 + \nu_1 e_1, \dots, x_k + \nu_k e_k\Big) - \phi\Big(z_1 + \nu_1 \mu_1e_1, \dots,  z_k + \nu_k \mu_k e_k\Big)\\
&-\sum_{\emptyset \not=I \subsetneq [k]} \Big(\prod_{i \in [k]} \mu_i\Big)\phi\Big((\nu_i e_{i})_{i \in I}, (x_{i})_{i \in [k] \setminus I}\Big)\\
&+\sum_{\emptyset \not=I \subsetneq [k]} \phi\Big((\nu_i\mu_{i}e_{i})_{i \in I}, (z_{i})_{i \in [k] \setminus I}\Big).
\end{align*}
Our task is to prove that these two expressions are equal. Hence, it suffices to prove that
\begin{align*}\phi\Big(z_1 + \tau_1e_1, \dots,  z_k + \tau_ke_k\Big) &- \sum_{\emptyset \not=I \subsetneq [k]} \phi\Big((\tau_{i}e_{i})_{i \in I}, (z_{i})_{i \in [k] \setminus I}\Big),\\
&=\phi\Big(z_1 + e_1, \dots,  z_k + e_k\Big) - \sum_{\emptyset \not=I \subsetneq [k]} \phi\Big((e_{i})_{i \in I}, (z_{i})_{i \in [k] \setminus I}\Big),\end{align*} 
where $\tau_i = \mu_i \lambda_i^{-1}$. Since  $\prod_{i \in [k]} \tau_i = 1$, this follows from Proposition~\ref{zeCancelation}.\\ [6pt]


\noindent\textbf{3.2. The extension map is multilinear.} 
\smallskip

Let $x_{[k]}, y_{[k]} \in B$ be arbitrary points that differ in a single coordinate. We need to show that $\phi^{\text{ext}}(x_{[k]}) - \phi^{\text{ext}}(y_{[k]}) = \phi^{\text{ext}}((x\ominus y)_{[k]})$. To begin, we show that $\phi^{\text{ext}}$ respects scalar multiplication in a single coordinate.

\begin{claim}Let $x_{[k]} \in B$ and let $\lambda \in \mathbb{F}$. Then
\[\phi^{\text{ext}}(x_1, \dots, x_{i-1}, \lambda x_i, x_{i+1}, \dots, x_k) = \lambda \phi^{\text{ext}}(x_{[k]}).\]\end{claim}

\begin{proof}If $x_{[k]} \in B^0$ or $\lambda = 0$, we are done, so assume the contrary. By Corollary~\ref{layerPath}, there is a $(2k+1)(2^k-1)$-good sequence $q^0_{[k]} = z_{[k]},q^1_{[k]}, \dots, q^s_{[k]} = (\lambda_1 x_1, \dots, \lambda_k x_k)$. Recall from~\eqref{phiDef} that $\phi^{\text{ext}}$ is defined by the formula
\[\phi^{\text{ext}}(x_{[k]}) = \rho(x_{[k]}) \Big(\phi(q^s_{[k]} \ominus q^{s-1}_{[k]}) + \dots + \phi(q^1_{[k]} \ominus q^0_{[k]}) + h_0.\Big)\]
Noting that the same $\mathbf{s}$-good sequence can be used for $(x_1, \dots, x_{i-1}, \lambda x_i, x_{i+1}, \dots, x_k)$, we find that
\[\phi^{\text{ext}}(x_1, \dots, x_{i-1}, \lambda x_i, x_{i+1}, \dots, x_k) = \lambda\rho(x_{[k]}) \Big(\phi(q^s_{[k]} - q^{s-1}_{[k]}) + \dots + \phi(q^1_{[k]} - q^0_{[k]}) + h_0\Big),\]
so the claim follows.\end{proof}

To finish the proof that $\phi^{\text{ext}}$ is multilinear, we distinguish two cases.\\

\noindent\emph{\textbf{Case 1:} at least one of the points $x_{[k]}, y_{[k]}, (x\ominus y)_{[k]}$ is in $B^0$.} 
\smallskip

Observe that $(x\ominus(x\ominus y))_{[k]}=y_{[k]}$, and also that $(y\ominus x)_{[k]}$ is equal to $(x\ominus y)_{[k]}$ except in the coordinate where $x$ and $y$ differ, which changes sign. Combining these observations and using the claim above, we may assume without loss of generality that $(x\ominus y)_{[k]} \in B^0$, which is equivalent to the statement that $\rho(x_{[k]}) = \rho(y_{[k]})$. If $\rho(x_{[k]}) = \rho(y_{[k]}) = 0$, then the map at all three points equals $\phi$, which we know to be multilinear. Hence, we may assume that $\rho(x_{[k]}) = \rho(y_{[k]}) \not= 0$. By Corollary~\ref{layerPath} there is a $(2k+1)(2^k-1)$-good sequence $q^0_{[k]} = z_{[k]}, \dots, q^s = (\lambda_1 y_1, \dots, \lambda_k y_k)$. But if we add the point $q^{s+1}_{[k]} = (\lambda_1 x_1, \dots, \lambda_k x_k)$, then we get an $\mathbf{s}$-good sequence for $x_{[k]}$ as well, so 
\begin{align*}\phi^{\text{ext}}(x_{[k]}) = &\rho(x_{[k]}) \Big(\phi(q^{s+1}_{[k]} \ominus q^{s}_{[k]}) + \phi(q^s_{[k]} \ominus q^{s-1}_{[k]}) + \dots + \phi(q^1_{[k]} \ominus q^0_{[k]}) + h_0\Big)\\
 = &\rho(x_{[k]}) \Big(\prod_{i \in [k]} \lambda_i\Big) \phi(x_{[k]} \ominus y_{[k]}) + \rho(y_{[k]}) \Big(\phi(q^s_{[k]} \ominus q^{s-1}_{[k]}) + \dots + \phi(q^1_{[k]} \ominus q^0_{[k]}) + h_0\Big)\\
= & \phi^{\text{ext}}(x_{[k]} \ominus y_{[k]}) + \phi^{\text{ext}}(y_{[k]}).\end{align*}

\noindent\emph{\textbf{Case 2:} no point belongs $B^0$.} 
\smallskip

In this case, we have that $\rho(x_{[k]}), \rho(y_{[k]}), \rho((x\ominus y)_{[k]}) \not= 0$. Let $d$ be the coordinate in which $x_{[k]}$ and $y_{[k]}$ differ. By Corollary~\ref{layerPath}, there is a $(2k+1)(2^k-1)$-good sequence $q^0_{[k]} = z_{[k]}, q^1_{[k]},\dots, q^s_{[k]} = (\lambda_1 x_1, \dots, \lambda_k x_k)$. Define points 
\[p^1_{[k]} = (\lambda_1 x_1, \dots, \lambda_{d-1} x_{d-1}, \mu y_d, \lambda_{d+1} x_{d+1}, \dots, \lambda_{k} x_{k})\] 
and 
\[p^2_{[k]} = (\lambda_1 x_1, \dots, \lambda_{d-1} x_{d-1}, \nu (x_d - y_d), \lambda_{d+1} x_{d+1}, \dots, \lambda_{k} x_{k}),\] 
where $\lambda, \mu$ are such that $\rho(y_{[k]}) = \mu^{-1} \prod_{i \in [k] \setminus \{d\}} \lambda_i^{-1}$ and $\rho((x\ominus y)_{[k]}) = \nu^{-1} \prod_{i \in [k] \setminus \{d\}} \lambda_i^{-1}$. The sequences $q^0_{[k]}, \dots,$ $q^s_{[k]},$ $p^1_{[k]}$ and $q^0_{[k]}, \dots,$ $q^s_{[k]},$ $p^2_{[k]}$ are also $\mathbf{s}$-good, so
\begin{align*}\phi^{\text{ext}}(x_{[k]}) = &\rho(x_{[k]}) \Big(\phi(q^s_{[k]} \ominus q^{s-1}_{[k]}) + \dots + \phi(q^1_{[k]} \ominus q^0_{[k]}) + h_0\Big),\\
\phi^{\text{ext}}(y_{[k]}) = &\rho(y_{[k]}) \Big(\phi(p^1_{[k]} \ominus q^s_{[k]}) + \phi(q^s_{[k]} \ominus q^{s-1}_{[k]}) + \dots + \phi(q^1_{[k]} \ominus q^0_{[k]}) + h_0\Big),\text{ and}\\
\phi^{\text{ext}}((x\ominus y)_{[k]}) = &\rho((x\ominus y)_{[k]}) \Big(\phi(p^2_{[k]} \ominus q^s_{[k]}) + \phi(q^s_{[k]} \ominus q^{s-1}_{[k]}) + \dots + \phi(q^1_{[k]} \ominus q^0_{[k]}) + h_0\Big).\\
\end{align*}
Hence, writing $\Lambda = \prod_{i \in [k] \setminus \{d\}} \lambda_i^{-1}$ and recalling that $\lambda_d \rho(x_{[k]}) = \mu \rho(y_{[k]}) = \nu \rho((x-y)_{[k]}) = \Lambda$, we have
\begin{align*}\phi^{\text{ext}}&(y_{[k]}) + \phi^{\text{ext}}((x\ominus y)_{[k]}) - \phi^{\text{ext}}(x_{[k]})\\ &=\rho(y_{[k]}) \phi(p^1_{[k]} \ominus q^s_{[k]}) + \rho((x\ominus y)_{[k]}) \phi(p^2_{[k]} \ominus q^s_{[k]})\\
&= \rho(y_{[k]}) \phi(\lambda_1 x_1, \dots, \lambda_{d-1} x_{d-1}, \mu y_d - \lambda_d x_d, \lambda_{d+1} x_{d+1}, \dots, \lambda_{k} x_{k})\\
&\hspace{2cm}+ \rho((x\ominus y)_{[k]}) \phi(\lambda_1 x_1, \dots, \lambda_{d-1} x_{d-1}, \nu (x_d -y_d) - \lambda_d x_d, \lambda_{d+1} x_{d+1}, \dots, \lambda_{k} x_{k})\\
&=\phi(\lambda_1 x_1, \dots, \lambda_{d-1} x_{d-1}, \rho(y_{[k]})(\mu y_d - \lambda_d x_d) + \rho((x\ominus y)_{[k]})(\nu (x_d -y_d) - \lambda_d x_d), \lambda_{d+1} x_{d+1}, \dots, \lambda_{k} x_{k})\\
&=\phi(\lambda_1 x_1, \dots, \lambda_{d-1} x_{d-1}, (\rho(y_{[k]})\mu) y_d + (\rho((x\ominus y)_{[k]})\nu) (x_d -y_d) - (\rho(x_{[k]}) \lambda_d) x_d, \lambda_{d+1} x_{d+1}, \dots, \lambda_{k} x_{k})\\
&=\phi\Big(\lambda_1 x_1, \dots, \lambda_{d-1} x_{d-1}, \Lambda \Big(y_d + (x_d -y_d) - x_d\Big), \lambda_{d+1} x_{d+1}, \dots, \lambda_{k} x_{k}\Big) = 0,
\end{align*}
completing the proof.
\bigskip

\boldSection{From multilinear maps on general varieties to global multilinear maps}

We are now ready to prove the main result, which will follow from the following proposition.

\begin{proposition}\label{mainThmIndStep}Let $\emptyset \in \mathcal{F} \subset \mathcal{P}[k]$ be a down-set\footnote{A collection of sets closed under taking subsets.} with a maximal set $S$. There are constants $C = C_{\mathcal{F}}, D = D_{\mathcal{F}}$ such that the following is true.\\
\indent Let $\beta_i \colon G_{I_i} \to \mathbb{F}$ be multilinear maps for $i \in [m]$, with $I_i \in \mathcal{F}$. Let $B = \{x_{[k]} \in G_{[k]} \colon (\forall i\in [m])\ \beta_i(x_{I_i}) = 0\}$ and let $\phi \colon B \to H$ be a multilinear map to a $\mathbb{F}$-vector space $H$. Then there exist $r \leq C m^D$, multilinear forms $\gamma_i \colon G_{J_i} \to \mathbb{F}$, $J_i \in \mathcal{F} \setminus \{S\}$, $i \in [r]$, and a multilinear map $\Phi \colon \{x_{[k]} \in G_{[k]} \colon (\forall i \in [r])\ \gamma_i(x_{J_i}) = 0\} \to H$ such that $\phi = \Phi$ on $\dom \phi \cap \dom \Phi$, where $\dom$ stands for the domain of a given function.\end{proposition}

\begin{proof}[Proposition \ref{mainThmIndStep} implies Theorem~\ref{mainThm}]Let $\mathcal{F}_1 = \mathcal{P}[k] \supsetneq \mathcal{F}_2 \supsetneq \dots \supsetneq \mathcal{F}_{2^k} = \{\emptyset\}$ be a sequence of down-sets in $\mathcal{P}[k]$, where we remove a maximal set $S_i$ from each down-set $\mathcal{F}_i$ to obtain the next one. Apply Proposition~\ref{mainThmIndStep} to $\mathcal{F}_1, S_1$ and $\phi$ to get a new multilinear map $\phi^1$ such that $\phi = \phi^1$ on $\dom \phi \cap \dom \phi^1$. Then, apply Proposition~\ref{mainThmIndStep} to $\mathcal{F}_2, S_2$ and $\phi^1$ to get another multilinear map $\phi^2$ such that $\phi^1 = \phi^2$ on $\dom \phi^1 \cap \dom \phi^2$ and proceed like this. The final map we get $\Phi = \phi^{2^k}$ is then a global multilinear map, and $\phi = \phi^{2^k}$ holds on $\dom \phi \cap \dom \phi^1 \cap \dots \cap \dom \phi^{2^k}$, which a multilinear variety of the codimension claimed in Theorem~\ref{mainThm}.\end{proof}

\begin{proof}[Proof of Proposition~\ref{mainThmIndStep}]Reordering the maps if necessary, we may assume that $I_1 = \dots = I_s = S$ and $I_{s+1}, \dots, I_m \not= S$. Let $\lambda^1, \dots, \lambda^n\in \mathbb{F}^s$ be a maximal independent sequence such that for each $i \in [n]$
\[\ex_{x_S} \chi\Big(\sum_{j \in [s]} \lambda^i_j \beta_j(x_S)\Big) \geq \frac{1}{2k^2}\mathbf{f}^{-(2k^2 + k + 1)(m+1)2^{2k+3}}.\]
Extend $\lambda^1, \dots, \lambda^n$ with $\mu^1, \dots, \mu^{s-n}$ to a basis of $\mathbb{F}^s$. Write $\rho_i = \sum_{j \in [s]} \mu^i_j \beta_j$ for $i \in [s-n]$ and $\alpha_i = \sum_{j \in [s]} \lambda^i_j \beta_j$ for $i \in [n]$. Then $\phi$ is defined on 
\begin{align*}\{x_{[k]} \in G_{[k]} \colon (\forall i \in [s-n])\ \rho_i(x_S) = 0\} \cap \{x_{[k]} &\in G_{[k]} \colon (\forall i \in [n])\ \alpha_i(x_S)= 0\}\\
 &\cap \{x_{[k]} \in G_{[k]} \colon (\forall i \in [s + 1, m])\ \beta_i(x_{I_i}) = 0\}.\end{align*}
Let $\mathcal{I} = \{i \in [s+1, m] \colon I_i \subset [k] \setminus S\}$. Then by Lemma~\ref{mlBias} the maps satisfy 
\begin{align*}\Big|\ex_{x_{S}} \chi\Big(\sum_{i \in [s-n]} \nu_i \rho_i(x_S) + \sum_{i \in [n]} \tau_i \alpha_i(x_S) + \sum_{i \in [s+1,m] \setminus \mathcal{I}} \sigma_i \beta_i(x_{S \cap I_i})\Big)\Big| \leq &\Big|\ex_{x_{S}} \chi\Big(\sum_{i \in [s-n]} \nu_i \rho_i(x_S) + \sum_{i \in [n]} \tau_i \alpha_i(x_S)\Big)\Big|\\
 < &\frac{1}{2k^2}\mathbf{f}^{-(2k^2 + k + 1)(m+1)2^{2k+3}}\end{align*}
when $\nu \in \mathbb{F}^{[s-n]} \setminus \{0\}, \tau \in \mathbb{F}^n, \sigma \in \mathbb{F}^{[s+1,m] \setminus \mathcal{I}}$ and 
\begin{equation}(\forall i \in [n])\ex_{x_{S}} \chi\Big(\alpha_i(x_S)\Big) \geq \frac{1}{2k^2}\mathbf{f}^{-(2k^2 + k + 1)(m+1)2^{2k+3}}.\label{structMap}\end{equation}

\begin{claim*}For $i \in [0,s-n]$, there is a multilinear variety $B^i \subset G_{[k] \setminus S}$ of codimension at most $im$ defined by maps whose coordinate sets belong to $\mathcal{F} \setminus \{S\}$, and a multilinear map $\psi^i \colon \dom \psi^i \to H,$ where
\begin{align*}\dom \psi^i = (B^i \times G_S) \cap \{x_{[k]} \in G_{[k]} \colon (\forall j \in [i+1, s&-n])\ \rho_j(x_S) = 0\} \cap \{x_{[k]} \in G_{[k]} \colon (\forall j \in [n])\ \alpha_j(x_S) = 0\}\\
& \cap \{x_{[k]} \in G_{[k]} \colon (\forall j \in [s+1,m])\ \beta_j(x_{I_j}) = 0\}\end{align*}
such that $\psi^i = \phi$ on $\dom \phi \cap \dom \psi^i$.\end{claim*}

\begin{proof}[Proof of claim] We argue by induction on $i$. The base case is $i = 0$, where we may take $\psi^0 = \phi$. Assume now that the claim holds for some $i-1 < s- n$, and let $B^{i-1}$ and $\psi^{i-1}$ be the corresponding variety and map. Take an arbitrary $z_S \in G_S$ such that $\rho_{i}(z_S) = 1$, $\rho_j(z_S) = 0$ for $j > i$, $\alpha_{j}(z_S) = 0$ for $j \in [n]$, and $\beta_j(z_{I_j}) = 0$ for $I_j \subset S$. Such a point exists by Lemma~\ref{nonzeroPtLemma}.\\
\indent We define $B^i = \{x_{[k] \setminus S} \colon (\forall j \in [s + 1, m] \colon I_j \not\subset S)\  \beta_j(x_{I_j \setminus S}, z_{I_j \cap S}) = 0\} \cap B^{i-1}$. Notice that coordinate sets of the maps defining $B^i$ lie in $\mathcal{F} \setminus \{S\}$. Next, define $\psi^i \colon \dom \psi^i \to H$, where
\begin{align*}\dom \psi^i = (B^i &\times G_S) \cap \{x_{[k]} \in G_{[k]}\colon (\forall j \in [i+1, s-n])\rho_j(x_S) = 0\}\\
& \cap \{x_{[k]} \in G_{[k]} \colon (\forall j \in [n])\ \alpha_j(x_S) = 0\}\cap \{x_{[k]} \in G_{[k]} \colon (\forall j \in [s+1,m])\ \beta_j(x_{I_j}) = 0\},\end{align*}
by extending the map $\tau_{x_{[k] \setminus S}} \colon (\dom \psi^{i-1})_{x_{[k] \setminus S}} \to H$ defined by $y_S \mapsto \psi^{i-1}(x_{[k] \setminus S}, y_S)$ by mapping $z_S$ to 0, for each $x_{[k] \setminus S} \in B^{i}$. Note that $\dom \psi^i$ has the property that $\dom \psi^i = \bigcup_{x_{[k] \setminus S} \in B^{i}} \{x_{[k] \setminus S}\} \times (\dom \psi^{i})_{x_{[k] \setminus S}}	$ and for each $x_{[k] \setminus S} \in B^{i}$, $z_S \in (\dom \psi^{i})_{x_{[k] \setminus S}} \setminus (\dom \psi^{i-1})_{x_{[k] \setminus S}}$. 
\indent By Theorem~\ref{qrExtension}, for each $x_{[k] \setminus S} \in B^i$ there is a unique multilinear extension $\theta_{x_{[k] \setminus S}}$ that satisfies this. We thus define $\psi^i$ for $(x_{[k] \setminus S}, y_S) \in \dom \psi^i$, by setting $\psi^i(x_{[k] \setminus S}, y_S) = \theta_{x_{[k] \setminus S}} (y_S)$.\\
\indent It suffices to show that $\psi^i$ is multilinear in directions $[k] \setminus S$. To this end, fix some $d \in [k] \setminus S$ and take $x^{1}_{[k] \setminus S}, x^{2}_{[k] \setminus S}, x^{3}_{[k] \setminus S} \in B^{i}$ which differ in coordinate $d$ and $x^1_d - x^2_d = x^3_d$. Write $D_2 = \cap_{j \in [3]} \dom \theta_{x^j_{[k] \setminus S}}$ and $D_1 = D_2 \cap \{y_S \in G_S \colon \rho_i(y_S) = 0\}$. Observe that $\theta_{x^1_{[k] \setminus S}} - \theta_{x^2_{[k] \setminus S}}$ is a multilinear map that extends $\tau_{x^1_{[k] \setminus S}} - \tau_{x^2_{[k] \setminus S}}$ from $D_1$ to $D_2$ and maps $z_S$ to 0. Also, $\theta_{x^3_{[k] \setminus S}}$ is a multilinear map that extends $\tau_{x^3_{[k] \setminus S}}$ from $D^1$ to $D_2$ and maps $z_S$ to 0. But $\tau_{x^1_{[k] \setminus S}} - \tau_{x^2_{[k] \setminus S}} = \tau_{x^3_{[k] \setminus S}}$ on $D_1$, so by uniqueness of extensions in Theorem~\ref{qrExtension}, we have $\theta_{x^1_{[k] \setminus S}} - \theta_{x^2_{[k] \setminus S}} = \theta_{x^3_{[k] \setminus S}}$ on $D_2$, as desired.\end{proof}

Apply the claim above with $i = s-n$. After that, it remains to remove maps $\alpha_{[n]}$. From~\eqref{structMap} and Theorem~\ref{strongInvThm}, we may find $m' \leq m C^{\text{ranks}}_k \Big(\Big((2k^2 + k + 1)(m+1)2^{2k + 3} + 2k^2\Big)^{D^{\text{ranks}}_k} + 1\Big)$ and further multilinear forms $\gamma_j \colon G_{J_j} \to \mathbb{F}$, $J_j \in \mathcal{F} \setminus \{S\}$, $j \in [m']$ such that 
\[\Big\{x_{[k]} \in G_{[k]} \colon (\forall j \in [m'])\ \gamma_j(x_{J_j}) = 0\Big\} \subseteq \Big\{x_{[k]} \in G_{[k]} \colon (\forall j \in [n])\ \alpha_j(x_S) = 0\Big\}.\]
Hence, the map $\Phi$ with domain
\begin{align*}\dom \Phi = (B^{s-n} \times G_S) \cap \{x_{[k]} \in G_{[k]} &\colon (\forall j \in [m']) \gamma_j(x_{J_j}) = 0\}\\
& \cap \{x_{[k]} \in G_{[k]} \colon (\forall j \in [s+1,m]) \beta_j(x_{I_j}) = 0\}\end{align*}
and $\Phi = \psi^{s-n}$ on its domain is the desired map. Its domain $\dom \Phi$ has codimension at most $m^2 + m + m' = O_k(m^{O_k(1)})$, which is the claimed bound. This completes the proof of Proposition~\ref{mainThmIndStep} and with it the proof of Theorem~\ref{mainThm}. \end{proof}

\end{document}